\documentclass[11pt,reqno]{article}
\usepackage{amsfonts}
\usepackage{amsmath}
\usepackage{amssymb}
\usepackage{epsfig}
\usepackage{amsthm}
\usepackage[english]{babel}
\usepackage{hyperref}
\usepackage{enumerate}
\usepackage{url}
\usepackage{authblk}
\usepackage[margin=15mm,hang]{caption}
\usepackage{subcaption}

\newcommand{\N}{\mathbb N}

\newcommand{\R}{\mathbb R}

\newcommand{\E}{\mathbb E}
\newcommand{\pr}{\mathbb P}

\newcommand{\salj}{\mathcal{F}}

\renewcommand{\labelenumi}{(\alph{enumi})}
\begin{document}
\theoremstyle{plain}
\newtheorem{thm}{Theorem}[section]
\newtheorem{lem}[thm]{Lemma}
\newtheorem{cor}[thm]{Corollary}
\newtheorem{prop}[thm]{Proposition}
\newtheorem{definition}{Definition}

\title{Phase transitions in non-linear urns with interacting types}
\author{Marcelo Costa \\ Universidad de Buenos Aires \thanks{\href{mailto:mrocha@dm.uba.ar}{mrocha@dm.uba.ar}} \and Jonathan Jordan \\ University of Sheffield \thanks{\href{mailto:jonathan.jordan@sheffield.ac.uk}{jonathan.jordan@sheffield.ac.uk}}}
\maketitle

\begin{abstract}We investigate reinforced non-linear urns with interacting types, and show that where there are three interacting types there are phenomena which do not occur with two types. In a model with three types where the interactions between the types are symmetric, we show the existence of a double phase transition with three phases: as well as a phase with an almost sure limit where each of the three colours is equally represented and a phase with almost sure convergence to an asymmetric limit, which both occur with two types, there is also an intermediate phase where both symmetric and asymmetric limits are possible.  In a model with anti-symmetric interactions between the types, we show the existence of a phase where the proportions of the three colours cycle and do not converge to a limit, alongside a phase where the proportions of the three colours can converge to a limit where each of the three is equally represented.\end{abstract}

\section{Introduction and definitions}

In general terms, an urn model is a system containing a number of particles of different types
(often regarded as balls of different colours, for ease of visualisation).
At each time step, a set of particles is sampled from the system, whose contents are then altered depending on the sample which was drawn.
Pemantle \cite{pemantlesurvey} surveys several ways to approach this model framework.

This paper is limited to models with a single urn from which a single ball is drawn, its colour is noted and it is then returned to the urn along with one new ball of that same colour.
In addition, we introduce a graph-based interaction according to which the probability of choosing a ball of a given colour is reinforced not only by its own proportion in the urn, but also by the proportions of balls of other colours. 
Therefore, the interaction arises among balls of different colours, as opposed to the so-called interacting urn models consisting of  systems of multiple urns (e.g. Bena\"{i}m et al \cite{BBCL15}, Launay and Limic \cite {LL12}) in which different urns (each containing balls of different colours) interact with one another.  Our model is also different from the graph-based competition described in van der Hofstad et al \cite{van2016strongly}, where the colours correspond to edges of the graph, which compete with, as opposed to being reinforced by, other edges incident on the same vertices.

We now formally define our model.  Consider an urn containing balls of $d$ colours.
The vector $x(n) = (x_1(n),\ldots, x_d(n)) \in \N^d$ denotes the number of balls of each colour at time $n=0,1,2,...$.
The strength of the reinforcement is given by a positive real number $\beta>0$ and we denote by $x^{\beta}(n)$ the coordinate-wise $\beta$ power of the column vector $x(n)$.
The interaction is defined as follows.
Given a non-negative matrix $A = (a_{ij})_{i,j=1}^d$, define the column vector 

\begin{equation}
  \label{eq:potential}
  u(n) := A x^{\beta}(n)
\end{equation}
Let $\salj_n$ be the $\sigma$-algebra generated by the $x(m)$ for $0\leq m\leq n$ and write $u_i(n)$ for the $i$-th component of $u(n)$.
The transition probabilities are then
\begin{equation}
  \label{eq:transition}
\pr(x(n+1)-x(n)=\mathbf{e}_i \> | \> \salj_n)=\frac{u_i(n)}{\sum_{j=1}^{d} u_j(n)}, \quad i=1,\ldots,d, 
\end{equation}
where $\mathbf{e}_i$ is the unit vector in direction $i$.  That is, one ball is added to the urn at each time step, and the right hand side of \eqref{eq:transition} gives the probability that it is of colour $i$.

Now, let $n_0>0$ be the initial number of balls so that at time $n$ the urn contains $n+n_0$ balls.
Then, the proportion of balls of each colour is a process in the $(d-1)$-dimensional simplex $\Delta^{d-1}$ given by the vector
\begin{equation}
  \label{eq:proportion}
\bar{x}(n)=x(n)/(n+n_0).  
 \end{equation}
When $A$ is a multiple of the identity matrix (therefore, $\bar{x}(n)$ having no interaction) it is well-known (see Oliveira \cite{RO07}) that the process $\bar{x}$ undergoes a phase transition as follows.
For $\beta < 1$, the process converges almost surely to the `centre' of the simplex, that is, the asymptotic proportion of balls of each colour are all the same.
For $\beta = 1$, commonly referred to as the P\'olya urn model, the process converges almost surely to a non-trivial random variable supported in the interior of the simplex. 
For $\beta > 1$, the process converges almost surely to one of the corners of the simplex.
In this case that a single type dominates was proved by Khanin and Khanin \cite{khanin2001} following on from the two-type case which can be covered using Rubin's Theorem in Davis \cite{davis_rrw}.

For a two-colour urn model $d=2$ and symmetric interaction,
\[   A = 
 \left( 
 \begin{matrix}
   1 & a \\
   a & 1 
 \end{matrix}
\right), \quad a>0, \]
it was proved by the first author in Theorem 2.2.1, \cite{MCthesis}, that there was a phase transition as follows.
\begin{align}
  & (i) \quad  \text { if } \quad  \left(\tfrac{1-a}{1+a}\right) \beta \leq 1, \quad \text{then} \quad  \bar{x}(n) \rightarrow (\tfrac{1}{2},\tfrac{1}{2}) \quad a.s. \nonumber\\
  & (ii) \quad \text {if } \quad  \left(\tfrac{1-a}{1+a}\right) \beta  > 1,   \quad \text{then} \quad  \bar{x}(n) \rightarrow  \Psi \quad a.s., \nonumber
\end{align}
where $\Psi$ is a random vector supported on $\left\{ \left( \frac{1}{1+r}, \frac{r}{1+r} \right), \left( \frac{r}{1+r}, \frac{1}{1+r}\right)\right\}$
and $r:= r(a, \beta)$ is the unique root in $(0,1)$ of $\mathcal{P}_{a,\beta}(z) = az^{\beta+1}- z^{\beta}+ z -a = 0.$
In case (\emph{ii}),  $\pr[\bar{x}(n) \rightarrow (\frac{1}{2},\frac{1}{2})] = 0.$
Note that for $\beta=1$, the process $(u(n))_{n \geq 0}$ is a Friedman's urn model and statement (\emph{i}) yields $u(n)/(u_1(n)+u_2(n)) \rightarrow (\frac{1}{2}, \frac{1}{2})$ $a.s.$ as expected.

Some similar results appear in Laruelle and Pag\`{e}s \cite{laruelle2019}.  The definition of the model there allows for a more general \textit{skewing function} than $x^{\beta}$ as above, but in the $d=2$ case gives phase transitions similar to those in \cite{MCthesis}, and relates them to the eigenvalues of the \textit{generating matrix} $H$, which is related to our matrix $A$.  Furthermore they show that, for any $d$, for a concave skewing function (corresponding to $\beta<1$ in our model) and a bi-stochastic generating matrix there is almost sure convergence to the centre of the simplex, and they give conditions, related to the eigenvalues of the generating matrix, under which for a convex skewing function (our $\beta>1$) there is probability zero of converging to the centre of the simplex.

In this paper, we follow up on the results for $d=2$ in \cite{MCthesis} and those in \cite{laruelle2019}, with the aim to generalise from $d=2$ to larger values of $d$ and to see whether more types of behaviour emerge when this is done.  We show that this is indeed the case when $d=3$, where we consider two particular choices of $A$.  Our results can be seen as extensions of those of \cite{laruelle2019} in certain specific cases.

First of all, we consider a choice of $A$ with a symmetric interaction of the same strength $a$ for each pair of colours.    The following theorem shows that in this system there are three phases, as opposed to two when $d=2$; there are phases where there is almost sure convergence to a symmetric limit and where there is almost sure convergence to one of a number of asymmetric limits, which are analogues of the phases when $d=2$, but there is also an intermediate phase where both symmetric and asymmetric limits are possible.

\renewcommand{\labelenumi}{(\roman{enumi})}
\begin{thm}\label{sym3main}
Let $A$ be the matrix
\begin{equation}
  \begin{pmatrix} 1 & a & a \\ a & 1 & a \\ a & a & 1 \end{pmatrix}, \quad a > 0.
\end{equation}\begin{enumerate}
\item Fix $a<1$.  Then there exists $\beta_1(a)$ satisfying $1<\beta_1(a)<\frac{1+2a}{1-a}$, with $\beta_1(a)$ an increasing function of $a$ satisfying $\beta_1(a)\to\infty$ as $a\to 1$, and we have the following three phases.
\begin{enumerate}
\item \label{symlim} \emph{Symmetric limit almost surely.} If $\beta<\beta_1(a)$, then almost surely $\bar{x}(n) \to \left(\frac13,\frac13,\frac13\right)$.
\item \label{bothlim} \emph{Symmetric or asymmetric limit.} If $\beta_1(a)<\beta<\frac{1+2a}{1-a}$ then there exists $r_2>1$ such that almost surely $\bar{x}(n)$ converges to one of the four points in $\Delta^2$ given by $\left(\frac13,\frac13,\frac13\right)$, $\left(\frac{r_{2}}{2+r_{2}},\frac{1}{2+r_{2}},\frac{1}{2+r_{2}}\right)$,  $\left(\frac{1}{2+r_{2}},\frac{r_{2}}{2+r_{2}},\frac{1}{2+r_{2}}\right)$ and  $\left(\frac{1}{2+r_{2}},\frac{1}{2+r_{2}},\frac{r_{2}}{2+r_{2}}\right)$.  All of these points have positive probability of being limits.
\item \label{asymlim} \emph{Asymmetric limit almost surely.} If $\beta>\frac{1+2a}{1-a}$ then there exists $r_{+}>1$ such that almost surely $\bar{x}(n)$ converges to one of the three points in $\Delta^2$ given by $\left(\frac{r_{+}}{2+r_{+}},\frac{1}{2+r_{+}},\frac{1}{2+r_{+}}\right)$,  $\left(\frac{1}{2+r_{+}},\frac{r_{+}}{2+r_{+}},\frac{1}{2+r_{+}}\right)$ and  $\left(\frac{1}{2+r_{+}},\frac{1}{2+r_{+}},\frac{r_{+}}{2+r_{+}}\right)$.
\end{enumerate}
\item Fix $a\geq 1$.  Then almost surely $\bar{x}(n) \to \left(\frac13,\frac13,\frac13\right)$.
\end{enumerate}
\end{thm}

Theorem \ref{sym3main} presents the results in terms of phase transitions in $\beta$ with $a$ fixed.  However, because both $\beta_1(a)$ and $\frac{1+2a}{1-a}$ are increasing functions of $a$ which converge to $1$ as $a\to 0$ and to $\infty$ as $a\to 1$, it is also possible to see them as phase transitions in $a$ with $\beta>1$ fixed: if $a<\frac{\beta-1}{2+\beta}$ then we will be in case (c), if $\frac{\beta-1}{2+\beta}<a<\beta_1^{-1}(\beta)$ then we will be in case (b), and if $a>\beta_1^{-1}(\beta)$ we will be in case (a).

Theorem \ref{sym3main} can be seen as an extension of the results of Proposition 2.15 of Laruelle and Pag\`{e}s \cite{laruelle2019} in the case where the matrix $H=\frac{1}{1+2a}\begin{pmatrix}1 & a & a \\ a & 1 & a \\ a & a & 1\end{pmatrix}$ in that their result shows the non-convergence to $(\frac13,\frac13,\frac13)$ in the case (i)(c), as the second largest eigenvalue of $H$ is $\frac{1-a}{1+2a}$.  We can also state the condition $\beta>\frac{1+2a}{1-a}$ in terms of the eigenvalues of $A$: the right hand side can be seen as the ratio of the two largest eigenvalues.  It might be reasonable to conjecture that an extension to $d>3$ of Theorem \ref{sym3main} might involve a similar condition on the eigenvalues; however we note that the other phase transition does not appear to be related to the eigenvalues of $A$.  We discuss the question of what happens with $d>3$ further in Section \ref{gthree}.

We also consider a system where each colour is reinforced by itself and by one other, in a cyclic way.  For this system, the following theorem shows the existence of a phase transition between a phase with convergence with positive probability to a symmetric limit and a phase where there is no convergence to a limit and there is cycling behaviour.

\begin{thm}\label{cyclicmain}Let $A$ be the matrix
$$\begin{pmatrix} 1 & 1 & 0 \\ 0 & 1 & 1 \\ 1 & 0 & 1 \end{pmatrix}.$$
\begin{itemize} \item When $\beta<4$ there is positive probability that $\bar{X}(n)\to (1/3,1/3,1/3)$. \item When $\beta>4$, almost surely $\bar{X}(n)$ fails to converge, and the limit set is either a periodic orbit or a connected union of periodic orbits.\end{itemize}\end{thm}

In Section \ref{stochapprox} we discuss the stochastic approximation methods we use in the proofs, while the proofs themselves are in Section \ref{symmetric} for Theorem \ref{sym3main} and Section \ref{cyclic} for Theorem \ref{cyclicmain}.  In the final Section \ref{ex_sim}, we illustrate the results with some examples and simulations, including some examples beyond those covered by Theorems \ref{sym3main} and \ref{cyclicmain}.

\section{Stochastic approximation approach} \label{stochapprox}

In this section we introduce some of the stochastic approximation ideas which appear in our proofs.

The type of stochastic approximation process we will be interested in is a Robbins-Monro algorithm in $\R^d$, following section 4.2 of Bena\"{\i}m \cite{benaim}.  This is a stochastic process $(y(n))_{n\in \N}$, with natural filtration $(\salj_n)$, taking values in $\R^d$ which satisfies 
\begin{equation}
  \label{eq:SAAgeneral}
  y(n+1) - y(n) = \gamma_n(F(y(n)) + \xi_{n+1}),
\end{equation}
where $F: \R^d \to \R^d$ is a deterministic vector field, $\gamma_n$ is a step size satisfying certain conditions, and $\E(\xi_{n+1}|\salj_n)=0$ with $\xi_n$ being $\salj_n$-measurable.  For most results it is required that $\gamma_n\to 0$ as $n\to\infty$ but that $\sum_{n=1}^{\infty}\gamma_n=\infty$, and some conditions are also needed on the noise term $\xi_{n+1}$, typically that it is bounded in $L^q$ for some $q$ which depends on the $\gamma_n$: see for example Proposition 4.2 of Bena\"{\i}m \cite{benaim}.

The above sequence can be thought as a numerical approximating method with varying step size $\gamma_n$ for solving the ODE $dx/dt = F(x)$.  Under the conditions discussed above, the asymptotic behavior of $(\bar{x}_n)_{n \in \N}$ and the underlying ODE are closely connected: as described in \cite{benaim}, define an interpolated version $(X(t))_{t \geq 0}$ by defining $\tau_0=0$ and for $n\in \N$ $\tau_n=\sum_{i=1}^n \gamma_i$, then defining $X(\tau_n+s)=\bar{x}_n+\frac{\bar{x}_{n+1}-\bar{x_n}}{\gamma_{n+1}}$ for $0\leq s < \gamma_{n+1}$.  Then the interpolated process $(X(t))_{t \geq 0}$ is an \emph{asymptotic pseudotrajectory} for the ODE.
This is called the \emph{ODE method} or \emph{the dynamical system approach}, which alongside some probabilistic techniques, is applied to examine almost sure dynamics of stochastic approximation processes.

Let $\Phi: \R_+ \times \R^d \to \R^d$ be the semiflow induced by the vector field $F$, so that $(\Phi(t,x))_{t\geq 0})$ is the trajectory of $F$ started at $x$.  A useful situation for analysis of stochastic approximation processes is where there is a \emph{Lyapunov function} for the vector field $F$: a function $V:\R^d \to \R^d$ such that on a trajectory $(\Phi(t,x))_{t\geq 0}$ of the vector field, $V(\Phi(t,x))$ is strictly decreasing in $t$ except where $\Phi(t,x)$ is a stationary point of $F$.  If this holds, then under mild conditions then the Robbins-Monro algorithm will converge almost surely to a (possibly random) limit, which will be a stationary point of $F$.  See section 6.2 of Bena\"{\i}m \cite{benaim}, and in particular Corollary 6.6 therein.

We now show that our process $(\bar{x}_n)$ can be put into Robbins-Monro form, and that the assumptions of Proposition 4.2 of Bena\"{\i}m \cite{benaim} are satisfied, allowing the theory of asymptotic pseudotrajectories to be used.  For a general matrix $A$ and a given configuration of balls $x(n)$ at time $n$, let $i_{n+1} \in \{1,\ldots,d\}$ be the random colour of the ball to be added in the urn at time $n+1$.
Then, note that
\begin{align}
  \bar{x}(n+1) & = \frac{x(n) + {\bf e}_{i_{n+1}}}{n_0 + n+1} = \frac{(n_0+n )\bar{x}(n) + {\bf e}_{i_{n+1}}}{n_0+n+1} \nonumber \\ 
  & = \left( 1 -\frac{1}{n_0+n+1} \right)\bar{x}(n) + \frac{{\bf e}_{i_{n+1}}}{n_0+n+1},
\end{align}
implying
\begin{equation}
  \label{eq:SAA}
\bar{x}(n+1) - \bar{x}(n) = \frac{1}{n_0+n+1}({\bf e}_{i_{n+1}} - \bar{x}(n)). 
\end{equation}
To put \eqref{eq:SAA} into Robbins-Monro form, we rearrange the right-hand side 
into a deterministic part and a zero mean ``noise''.
More specifically, let
\begin{equation}
  \label{eq:VFF}
  F(\bar{x}(n)) := \E[ {\bf e}_{i_{n+1}} \, | \, \salj_n] - \bar{x}(n),
\end{equation}
and
\begin{equation}
  \label{eq:NOISE}
  \xi_{n+1} := {\bf e}_{i_{n+1}} - \E[ {\bf e}_{i_{n+1}} \, | \, \salj_n].    
\end{equation}  
By setting $\gamma_n = 1/(n_0+n+1)$, we obtain
\begin{equation}
  \label{eq:SAAX}
  \bar{x}(n+1) - \bar{x}(n) = \gamma_n(F(\bar{x}(n)) + \xi_{n+1}).
\end{equation}  We note that the components of $\xi_{n+1}$ are uniformly bounded in modulus by $2$, meaning that the conditions of Proposition 4.2 of Bena\'{\i}m \cite{benaim} are satisfied, meaning that the interpolated process $(X(t))_{t\geq 0}$ is indeed an asymptotic pseudotrajectory of $F$.
%

We can write the components of $F$ as
$$F_i(x)=\frac{u_i}{\sum_{j=1}^d u_j}-x_i, \quad i=1,\ldots,d,$$
where $u_i$ has the same relationship to $x$ as $u_i(n)$ to $x(n)$.

For a stochastic approximation heuristic, and to apply some of the results we use, it is necessary to classify the stationary points, for which we use the following terminology.
\begin{definition}\label{stable_defs}Consider a stationary point $p$ of $F$.  We will say that $p$ is \emph{stable} if it is an attractor for the vector field  $F$, meaning that there exists a neighbourhood $W$ of $p$ where for $x\in W$ the trajectory $(\Phi(t,x))_{t\geq 0}$ satisfies $\Phi(x,t)\to p$ as $t\to \infty$, uniformly in $W$.  Furthermore, if all eigenvalues of $DF(p)$ have negative real part, $p$ is said to be \emph{linearly stable}, while if some eigenvalue has positive real part, $p$ is said to be \emph{linearly unstable}.  If all eigenvalues have positive real part, then $p$ is said to be a \emph{source}.\end{definition}

Typically convergence happens with positive probability to a stable stationary point but with probability zero to a linearly unstable one.  For the former, Theorem 7.3 of Bena\"{\i}m \cite{benaim}, which holds under the assumptions on the noise of Proposition 4.2, shows that a stable stationary point has positive probability of being a limit as long as its basin of attraction $W$ contains points which are \emph{attainable} in the sense that the process has positive probability of being indefinitely close to them at indefinitely large times, a condition which is usually satisfied for all points in the simplex for urn type processes under mild conditions; see Example 7.2 of \cite{benaim}.

To show that linearly unstable stationary points have zero probability of being limits it is usually necessary to check that there is expectation bounded away from zero of the positive part of the component of the noise in any given direction; see Pemantle \cite{pemantle1990}, Theorem 1.  This condition ensures that the process has zero probability of being trapped on a stable manifold and converging to the unstable point.  Theorem 1 of \cite{pemantle1990} also requires the noise term $\xi_{n+1}$ (in our notation) to be uniformly bounded; as noted above this is satisfied in our process.

\section{Proofs for the symmetric case}\label{symmetric}

In this section we prove Theorem \ref{sym3main}.

Throughout this section we let $A$ be the matrix
\begin{equation}
  \label{eq:matrix}
  \begin{pmatrix} 1 & a & a \\ a & 1 & a \\ a & a & 1 \end{pmatrix}, \quad a > 0.
\end{equation}

In this case the vector field $F$ is given by
\begin{eqnarray}F_1(x_1,x_2,x_3) &=& \frac{x_1^{\beta}+ax_2^{\beta}+ax_3^{\beta}}{(1+2a)(x_1^{\beta}+x_2^{\beta}+x_3^{\beta})}-x_1 \label{eq:1}\\
F_2(x_1,x_2,x_3) &=& \frac{ax_1^{\beta}+x_2^{\beta}+ax_3^{\beta}}{(1+2a)(x_1^{\beta}+x_2^{\beta}+x_3^{\beta})}-x_2 \label{eq:2}\\
F_3(x_1,x_2,x_3) &=& \frac{ax_1^{\beta}+ax_2^{\beta}+x_3^{\beta}}{(1+2a)(x_1^{\beta}+x_2^{\beta}+x_3^{\beta})}-x_3 \label{eq:3}\end{eqnarray}

The stochastic approximation approach indicates that the possible limits of our process will be stationary points of $F$, so we start by identifying these.  Noting that the lines $x_1=x_2$, $x_1=x_3$ and $x_2=x_3$ are each invariant under $F$, define the function
\begin{equation}
  \label{eq:pol}
\mathcal{P}_{a,\beta}(z) = az^{\beta+1} - z^{\beta} + (1+a)z -2a,  
\end{equation}
which we will see is related to the dynamics restricted to one of these lines. The following result shows that all stationary points of $F$ are located on at least one of these lines and expresses them in terms of solutions to $\mathcal{P}_{a,\beta}(z)=0$.

\begin{prop}\label{eqcoord}
All stationary points $(x_1,x_2,x_3)$ of $F$ have at least two of $x_1, x_2, x_3$ equal, and are of one of the forms $(\frac{r}{r+2},\frac{1}{r+2},\frac{1}{r+2})$, $(\frac{1}{r+2},\frac{r}{r+2},\frac{1}{r+2})$ or $(\frac{1}{r+2},\frac{1}{r+2},\frac{r}{r+2})$, with  $r$ a solution of $\mathcal{P}_{a,\beta}(z)=0$ in $\R^+$.

Furthermore, there are at most three possible values of $r$, one of which is always $1$, corresponding to the stationary point $(\frac13,\frac13,\frac13)$.
\end{prop}

\begin{proof}
We start off by showing that any stationary point has at least two co-ordinates equal.  We do this by writing the stationary point in the form $(x,rx,sx)$ and showing that one of $r=1$, $s=1$ or $r=s$ must hold.

Rearranging \eqref{eq:1}, \eqref{eq:2} and \eqref{eq:3} at $(x,rx,sx)$ gives \begin{eqnarray}\label{rex} x &=& \frac{1+a(r^{\beta}+s^{\beta})}{(2a+1)(1+r^{\beta}+s^{\beta})} \\  \label{rey} rx &=& \frac{r^{\beta}+a(1+s^{\beta})}{(2a+1)(1+r^{\beta}+s^{\beta})} \\ \label{rez} sx &=& \frac{s^{\beta}+a(1+s^{\beta})}{(2a+1)(1+r^{\beta}+s^{\beta})}.\end{eqnarray}
It follows that \begin{eqnarray}\label{req} r^{\beta}+a(1+s^{\beta})=r(1+a(r^{\beta}+s^{\beta})) \\ \label{seq} s^{\beta}+a(1+r^{\beta})=s(1+a(r^{\beta}+s^{\beta})).\end{eqnarray}
Take the linear combination $(s+\frac{1}{a})\times$\eqref{req}$-(r+1)\times$\eqref{seq}.  This eliminates $s^{\beta}$ and $s^{\beta+1}$, giving \begin{equation}\label{rearr3}(r^{\beta}+a)s-ar(1+r^{\beta})=(ar^{\beta}+1)s-\left(a(1+r^{\beta})+\frac{r}{a}-\frac{r^{\beta}}{a}+r^{\beta+1}-1\right),\end{equation} which can be rearranged to give \begin{equation}\label{rearr4}s(r^{\beta}-1)(a-1)=(r^{\beta+1}-1)(1-a)+(r^{\beta}-r)\left(a-\frac{1}{a}\right).\end{equation}  Assuming $a\neq 1$, \eqref{rearr4} gives $r=1$ or
\begin{equation}\label{sfromr}s=\frac{a(r^{\beta}+1)(1-r)+r^{\beta}-r}{a(r^{\beta}-1)}=\frac{\mathcal{P}_{a,\beta}(r)-ar^{\beta}+a}{a(1-r^{\beta})}.\end{equation}

Using this form for $s$ in \eqref{seq} gives (if $r\neq 1$) \begin{equation}\label{sbfromr}s^{\beta}=\frac{-ar^{\beta+1}+r^{\beta}+a-r}{a(r-1)}=\frac{\mathcal{P}_{a,\beta}(r)}{a(1-r)}+1.\end{equation}  Combining \eqref{sfromr} and \eqref{sbfromr} tells us that either $r=1$ or $s=1$ or $$\frac{s^{\beta}-1}{s-1}=\frac{r^{\beta}-1}{r-1},$$ and the latter case implies $r=s$.  Hence any stationary point has two co-ordinates equal.

We now assume, without loss of generality, that the stationary point is of the form $(x,x,rx)$ or equivalently $\left(\frac{1}{r+2},\frac{1}{r+2},\frac{r}{r+2}\right)$.  That the stationary point equations for a point of this form imply $\mathcal{P}_{a,\beta}(r)=0$ is easy to check, and it is also easy to check that $\mathcal{P}_{a,\beta}(1)=0$ for any $a,\beta>0$.

The function $\mathcal{P}_{a,\beta}$ satisfies $\mathcal{P}_{a,\beta}(0)<0$ and $\mathcal{P}_{a,\beta}(z)\to \infty$ as $z\to\infty$; furthermore it is concave for $z<\frac{\beta-1}{a(\beta+1)}$ and convex for $z>\frac{\beta-1}{a(\beta+1)}$, which indicates that it has either one root or three in $\R^{+}$, counting multiplicity.  This completes the proof.

\end{proof}

We now show that the process will indeed, almost surely, converge to one of the stationary points identified in Proposition \ref{eqcoord}.

\begin{lem}\label{point}
  The limit set of the process $(\bar{x}(n))_{n\in\N}$, defined in \eqref{eq:proportion} with given matrix \eqref{eq:matrix}, will, almost surely, be a single point which is a stationary point of $F$.
\end{lem}
\begin{proof}
  Let
  \begin{equation}
    \label{eq:lyapunov}
    L(x_1,x_2,x_3) = (x_1+x_2+x_3) - \frac{1}{2a+1} \left[ a\log(x_1x_2x_3) - \frac{1}{\beta}(a-1)\log(x_1^{\beta}+x_2^{\beta}+x_3^{\beta}) \right].  
  \end{equation}
  
  Then $L$ is a strict Lyapunov function for $F$.
  In fact, straightforward differentiation gives $$\frac{\partial L}{\partial x_i} = -\frac{1}{x_i} F_i.$$
  Then, denoting an trajectory of $F$ by $x(t) = (x_1(t), x_2(t), x_3(t))$, we obtain
  \[ \frac{d (L \circ x)}{d t} = \sum_{i=1}^3\frac{\partial L }{\partial x_i}\frac{d x_i}{d t} = -\sum_{i=1}^3x_i\left( \frac{\partial L}{\partial x_i}\right)^2 \leq 0,\]
where the equality holds in the above inequality if and only if $F(x)=0$.  Hence $L(x(t))$ is decreasing in $t$, and strictly decreasing except where $x(t)$ is a stationary point of $F$.

We now apply Corollary 6.6 of Bena\"{i}m \cite{benaim} to show that the limit set of $(\bar{x}(n))_{n\in\N}$ will almost
  surely be a stationary point of $F$.  As $\Delta^2$ is precompact and we have identified a strict Lyapunov function, the only condition which needs to be checked is that $F$ has countably many stationary points, which follows from Proposition \ref{eqcoord}.
 In fact, Proposition \ref{eqcoord} shows that $F$ has no connected sets of stationary points other than single points, so the limit set must be a single point, which is one of the stationary points of $F$.
\end{proof}

Note that the Lyapunov function $L$ generalises in an obvious way to more than three types, as long as all off-diagonal entries are equal.

\begin{prop}\label{Proots} Consider $\mathcal{P}_{a,\beta}(z) = az^{\beta+1} - z^{\beta} + (1+a)z -2a$ for $z\in \R^+$, with $a,\beta > 0$.
  \begin{enumerate}
  \item For a given value of $a>1$, $\mathcal{P}_{a,\beta}$ has only one root at $z=1$.\label{a>1}
  \item For a given value of $a<1$, there exists $\beta_1(a)$ satisfying $\frac{1+2a}{1-a}>\beta_1(a)>1$ such that
    \begin{enumerate}
      \item If $\beta>\frac{1+2a}{1-a}$ then $\mathcal{P}_{a,\beta}'(1)<0$ and we have that $\mathcal{P}_{a,\beta}$ has three roots in $\R^+$, $1$, $r_{-}$ and $r_{+}$, labelled so that $r_{-}<1<r_{+}$.  As functions of $\beta$ for fixed $a$, $r_{+}$ is increasing and $r_{-}$ is decreasing.
      \item If $\beta_1(a)<\beta<\frac{1+2a}{1-a}$ then $\mathcal{P}_{a,\beta}'(1)>0$ and $\mathcal{P}_{a,\beta}$ has three roots in $\R^+$, $1$, $r_1$ and $r_2$, labelled so that $1<r_1<r_2$.  As functions of $\beta$ for fixed $a$, $r_{2}$ is increasing and $r_{1}$ is decreasing.
      \item If $\beta<\beta_1(a)$ then $\mathcal{P}_{a,\beta}'(1)>0$ and the only root of $\mathcal{P}_{a,\beta}$ in $\R^+$ is $1$.
    \end{enumerate}
    Furthermore $\beta_1(a)$ is an increasing function of $a$ with $\beta_1(a) \to\infty$ as $a\to 1$.
  \end{enumerate}
\end{prop}
\begin{proof}
We start by observing that $\mathcal{P}_{a,\beta}(0)=-2a$ and that $\mathcal{P}_{a,\beta}(z)\to \infty$ as $z\to\infty$, indicating that $\mathcal{P}_{a,\beta}$ has an odd number of roots in $\R^+$, counting multiplicity.  Differentiating with respect to $z$, we have
\begin{equation}
\mathcal{P}_{a,\beta}'(z)=a(\beta+1)z^{\beta}-\beta z^{\beta-1}+1+a\end{equation} and \begin{equation}\mathcal{P}_{a,\beta}''(z)=z^{\beta-2}(a\beta(\beta+1)z-\beta(\beta-1)).\end{equation}  Because $\mathcal{P}_{a,\beta}''(z)$ is increasing on $z\in \R^+$ there cannot be more than three roots of $\mathcal{P}_{a,\beta}$ in $\R^+$.
  \emph{(i)} First, if $\beta<1$ we have that $\mathcal{P}_{a,\beta}''(z)>0$ in $\R^+$ implying that $\mathcal{P}_{a,\beta}'(z)$ is strictly increasing.
  Moreover, $\mathcal{P}_{a,\beta}'(z)$  goes from $-\infty$ to $+\infty$ when $z$ ranges from $0$ to $+\infty$ and $\mathcal{P}_{a,\beta}'(1) > 0 $.
  Then $\mathcal{P}_{a,\beta}'(z)$ changes sign only once at some $z^*<1$.
  Now, since $\mathcal{P}_{a,\beta}(0) = -2a < 0$ and $\mathcal{P}_{a,\beta}(z)$ is decreasing for $z<z^*<1$ and increasing otherwise, it follows that $\mathcal{P}_{a,\beta}(z)$ crosses the $z=0$ line only once at $z=1$.
  Second, the same happens for $\beta > 1$ since $\mathcal{P}_{a,\beta}'(z) > 0$ in $\R^+$ and so $\mathcal{P}_{a,\beta}(z)$ is strictly increasing.
  The case $\beta = 1 $ is trivial.

  \emph{(ii)} If $\beta>\frac{1+2a}{1-a}$ we have $\mathcal{P}_{a,\beta}'(1)=1+2a-\beta(1-a)<0$, indicating that in this case $\mathcal{P}_{a,\beta}$ must have three roots.
  Note that if $\beta=\frac{1+2a}{1-a}$ then $\mathcal{P}_{a,\beta}'(1)=0$ but that $\mathcal{P}_{a,\beta}''(1)<0$, showing that this is a double root, not a triple root, and so there must be another root in that case for larger $z$.
  If $\beta<\frac{1+2a}{1-a}$ then  $\mathcal{P}_{a,\beta}'(1)>0$ and $\mathcal{P}_{a,\beta}$ has no root less than $1$.
  Then, there must be either none, one double, or two distinct additional roots greater than 1.

The derivative of $\mathcal{P}_{a,\beta}(z)$ with respect to $\beta$, for fixed $a$ and $z$, is $(az-1)(\log z)z^{\beta}$.  Thus, for any $z\in(1,1/a)$, $\mathcal{P}_{a,\beta}(z)$ is decreasing in $\beta$, and it follows that if there are roots of $\mathcal{P}_{a,\beta}$ in this range for a particular value of $\beta$ there must also be for any larger $\beta$.  As $\mathcal{P}_{a,\beta}(z)>0$ if $z\geq 1/a$, this shows that there exists $\beta_1(a)\in[1,\frac{1+2a}{1-a}]$ such that there is one root of $\mathcal{P}_{a,\beta}$ when $\beta<\beta_1(a)$ and three when $\beta>\beta_1(a)$.

Let $\beta_0(a)>\frac{2}{1-a}-1>1$ be the unique solution to \begin{equation}\label{beta0}1+a=\left(\left(1-\frac{2}{\beta+1}\right)\frac{1}{a}\right)^{\beta-1}.\end{equation} (It can be seen that \eqref{beta0} has a unique solution for fixed $a<1$, as in that case the right hand side is increasing in $\beta$ if the right hand side is greater than $1$, is equal to $1$ at $\beta=1$ and tends to $\infty$ as $\beta\to\infty$.  That $\beta_0(a)>\frac{2}{1-a}-1$ can be seen by noting that if $\beta$ is a solution of \eqref{beta0} we must have $\left(1-\frac{2}{\beta+1}\right)\frac{1}{a}>1$.)   Then if $\beta<\beta_0(a)$ we have $\mathcal{P}_{a,\beta}'(z)>0$ for all $z$ and hence $\mathcal{P}_{a,\beta}$ is increasing in $z$ and so $z=1$ is the only root.  This shows that $\beta_1(a) \geq \beta_0(a)>1$.

Now, the fact that as mentioned above there is a root greater than $1$ when $\beta=\frac{1+2a}{1-a}$ together with the continuity of $\mathcal{P}_{a,\beta}(z)$ in $\beta$ ensures that there remains a root greater than $1$ for $\beta\in\left(\frac{1+2a}{1-a}-\epsilon,\frac{1+2a}{1-a}\right)$ for some $\epsilon>0$, so $\beta_1(a)<\frac{1+2a}{1-a}$.

The claims that $r_{+}$ and $r_2$ are increasing functions of $\beta$ and that $r_1$ is a decreasing function of $\beta$ also follow from the negative derivative of $\mathcal{P}_{a,\beta}(z)$ with respect to $\beta$ for $z\in(1,1/a)$.  Similarly the claim that $r_{-}$ is a decreasing function of $\beta$ follows from the derivative of $\mathcal{P}_{a,\beta}(z)$ with respect to $\beta$ being positive on $(0,1)$.

To see that  $\beta_1(a)$ is an increasing function of $a$, note that the derivative of $\mathcal{P}_{a,\beta}(z)$ with respect to $a$ is $z^{\beta+1}+z-2$, and for fixed $z>1$ this is positive, meaning that if we are in case (c) for particular choices of $a$ and $\beta$ we will also be in case (c) for the same value of $\beta$ and any larger value of $a$.  That $\beta_1(a)\to \infty$ as $a\to 1$ follows from $\beta_0(a)>\frac{2}{1-a}-1$.
\end{proof}

We now investigate the stability of these roots, for which recall the terminology in Definition \ref{stable_defs}.  

\begin{prop}\label{stable}
If a stationary point for $F$ is of the form $(x,x,rx)$ or $(x,rx,x)$ or $(rx,x,x)$, then it is linearly stable if $\mathcal{P}_{a,\beta}'(r)>0$ and $\frac{r^{\beta}+2}{r+2}>\frac{\beta(1-a)}{2a+1}$, and linearly unstable if either $\mathcal{P}_{a,\beta}'(r)<0$ or $\frac{r^{\beta}+2}{r+2}<\frac{\beta(1-a)}{2a+1}$.\end{prop}

\begin{proof}
Without loss of generality we focus on the case $(x,x,rx)$.

We note that the differential equation driven by $F$ keeps the line $x_1=x_2$ invariant, so we consider it restricted to this line;  the equation for $F_3$ gives $$F_3\left(\frac{1-x_3}{2},\frac{1-x_3}{2},x_3\right)=\frac{2a\left(\frac{1-x_3}{2}\right)^{\beta}+x_3^{\beta}}{(1+2a)\left(2\left(\frac{1-x_3}{2}\right)^{\beta}+x_3^{\beta}\right)}-x_3.$$  Let $x_3=z/(z+2)$ so that $x_1=x_2=1/(z+2)$.  Then $$F_3\left(\frac{1-x_3}{2},\frac{1-x_3}{2},x_3\right)=\frac{-2\mathcal{P}_{a,\beta}(z)}{(1+2a)(2+z^{\beta})},$$ and so is positive when $\mathcal{P}_{a,\beta}(z)$ is negative and vice versa.  Hence a stationary point $(x,x,rx)$ is stable in this direction if $\mathcal{P}_{a,\beta}'(r)>0$ and linearly unstable in this direction if $\mathcal{P}_{a,\beta}'(r)<0$.

Because $F$ is symmetric in $x_1$ and $x_2$, the other direction in which we need to consider stability will be perpendicular to this one.  Hence we consider
\begin{equation}\label{estability}
  \begin{split}
  F_1(x+\epsilon,x-\epsilon,rx) &= \frac{(x+\epsilon)^{\beta}+(x-\epsilon)^{\beta}+ar^{\beta}x^{\beta}}{(2a+1)((x+\epsilon)^{\beta}+(x-\epsilon)^{\beta}+r^{\beta}x^{\beta})}-x-\epsilon \\
                                &= F_1(x,x,rx)+\epsilon\left(-1+\frac{\beta(1-a)}{(2a+1)(2+r^{\beta})x}\right)+o(\epsilon) \\
                                &= F_1(x,x,rx)+\epsilon\left(-1+\frac{\beta(1-a)(r+2)}{(2a+1)(2+r^{\beta})}\right)+o(\epsilon).  
  \end{split}
\end{equation}

It follows that $(x,x,rx)$ is a stable stationary point in the direction perpendicular to the line $x_1=x_2$ if $\frac{r^{\beta}+2}{r+2}>\frac{\beta(1-a)}{2a+1}$ and linearly unstable in that direction if the reverse inequality applies.
\end{proof}

We shall henceforth restrict ourselves to the case $a<1$ since Propositions \ref{eqcoord}, \ref{Proots}(\ref{a>1}) and \ref{stable} imply that if $a>1$, $(\frac13,\frac13,\frac13)$ is the only stable stationary point for $F$ and by Lemma \ref{point}, $(\bar{x}(n))_{n \in \N}$ must converge to it.  The case $a=1$ has the probabilities of each colour being $1/3$ regardless of $\bar{x}(n)$ and so it is easily seen that $\bar{x}(n)\to (1/3,1/3,1/3)$ almost surely.

\begin{cor}\label{criteria}
  Assume $a<1$.
  \begin{enumerate}

  \item If $\beta<\beta_1(a)$, then the stationary point $(\frac13,\frac13,\frac13)$ is stable, and is the limit with probability $1$.
  \item If $\beta_1(a)<\beta<\frac{1+2a}{1-a}$, then the stationary points $(\frac13,\frac13,\frac13)$ and $\left(\frac{1}{r_{2}+2},\frac{1}{r_{2}+2},\frac{r_2}{r_{2}+2}\right)$ (and its permutations) are  stable, while the stationary point $\left(\frac{1}{r_{1}+2},\frac{1}{r_{1}+2},\frac{r_1}{r_{1}+2}\right)$ and its permutations are linearly unstable.
  \item If $\beta>\frac{1+2a}{1-a}$, then there are three stationary points of $F$ of the form $(x,x,rx)$ corresponding to the three solutions $r_{-}<1<r_{+}$ of $\mathcal{P}_{a,\beta}(z)=0$ in $\R^+$.  The stationary points $(\frac13,\frac13,\frac13)$ and $\left(\frac{1}{r_{-}+2},\frac{1}{r_{-}+2},\frac{2}{r_{-}+2}\right)$ (and its permutations) are linearly unstable, while $\left(\frac{1}{r_{+}+2},\frac{1}{r_{+}+2},\frac{r_{+}}{r_{+}+2}\right)$ and its permutations are stable.

  \end{enumerate}
  
\end{cor}

\begin{proof}
\emph{(i)}
  Stability follows from Proposition \ref{stable}, and almost sure convergence from Lemma \ref{point}.
  
\emph{(ii)}
  The shape of $\mathcal{P}_{a,\beta}$ as discussed in the proof of Proposition \ref{Proots} ensures that $r_1,r_2>1$ and that $\mathcal{P}_{a,\beta}'(r_1)<0$ and $\mathcal{P}_{a,\beta}'(r_2)>0$, showing that $\left(\frac{1}{r_{1}+2},\frac{1}{r_{1}+2},\frac{r_1}{r_{1}+2}\right)$ is linearly unstable, and that for the other two stationary points we just need to check the stability perpendicular to the line $x_1=x_2$.  But $\frac{r_2^{\beta}+2}{r_2+2}>1>\frac{\beta(1-a)}{2a+1}$ by our assumption on $\beta$, so the condition from Propostion \ref{stable} is satisfied and so $(\frac13,\frac13,\frac13)$ and $\left(\frac{1}{r_{2}+2},\frac{1}{r_{2}+2},\frac{r_2}{r_{2}+2}\right)$ are stable.
  \emph{(iii)}
  That $\mathcal{P}_{a,\beta}(z)=0$ has three solutions follows from Proposition \ref{Proots}. As $\mathcal{P}_{a,\beta}'(1)<0$ it follows from Propostion \ref{stable} that $(\frac13,\frac13,\frac13)$ is linearly unstable, and as $r_{-}<1$ we have $\frac{r_{-}^{\beta}+2}{r_{-}+2}<1<\frac{\beta(1-a)}{2a+1}$, so $\left(\frac{1}{r_{-}+2},\frac{1}{r_{-}+2},\frac{r_{-}}{r_{-}+2}\right)$ is also linearly unstable.  The global minimum of the Lyapunov function on $\Delta^2$ must be a stable stationary point of $F$, so the remaining stationary points, $\left(\frac{1}{r_{+}+2},\frac{1}{r_{+}+2},\frac{r_{+}}{r_{+}+2}\right)$ and its permutations, must be stable.
\end{proof}

%

The following result completes the proof of Theorem \ref{sym3main}.
\begin{prop}\label{convergence} Let $l(\bar{x})$ denote the limit set of the process $(\bar{x}(n))_{n \in \N}$ defined in \eqref{eq:SAAX} and recall the stability criteria in Proposition \ref{stable}. Then we have
  \begin{enumerate}
  \item $\pr[l(\bar{x}) = \{p\}] > 0$ for stable points $p$ of $F$.
  \item $\pr[l(\bar{x}) = \{p\}] = 0$ for linearly unstable points $p$ of $F$.
  \end{enumerate}
\end{prop}

\begin{proof}
  Without loss of generality we focus on the case $(x,x,rx)$.
  
  \emph{(i)} Let us now show that the process $\bar{x}$ in fact converges with positive probability toward a given stable fixed point.
  Of course, it is necessary that the process has positive probability of being arbitrarily close to the attractor at arbitrarily large times.
  That is, a point $p$ is said to be attainable by a process $X$ if for each $t>0$ we have that  $\pr[\exists ~ s \geq t ~:~ X(s) \in N_p] > 0$ for every neighborhood $N_p$ of $p$.
  It turns out that if the function $F + Id$ associated with an urn process $X$ maps the simplex into its interior, it follows that every point of the simplex is attainable by $X$.
  This is indeed the case for our process $\bar{x}$.
  Finally, Bena\"{i}m \cite{benaim} Theorem 7.3 ensures that a given attainable attractor $p$ with non-empty basin of attraction is such that $\pr[l(\bar{x}) = \{p\}] > 0$.

  \emph{(ii)} Let $p$ be a linearly unstable critical point in the interior of $\Delta^2$ and $\mathcal{N}_p \subset \Delta^2$ a neighborhood of $p$.
  The simplex is considered as a differential manifold by identifying its tangent space at any point with the linear subspace
  $T\Delta^2 = \{x \in \R^3 \> : \> \sum_i x_i = 0\}$.
  We need to check Pemantle's non-convergence criteria (Theorem 1 in \cite{pemantle1990}).  As we have $\gamma_n=1/(n_0+n+1)$ and we have bounded noise, the only condition we need to check is condition (6): that the expectation of the positive part of the component of the noise in any given direction is uniformly bounded away from zero.
  Formally, we need that whenever $\bar{x}_n \in \mathcal{N}_p$, there is a constant $\kappa$ such that
  $\E[\max\{\xi_{n+1} \cdot \theta,0\}  \> | \> \mathcal{F}_n] \geq \kappa$
  for every unit vector $\theta =(\theta_1,\theta_2,\theta_3) \in T\Delta^2$, where $\xi_{n+1}$ is the noise term from \eqref{eq:NOISE}.
  For notational simplicity, write $\tilde{u}_i = u_i(n)/(\sum_j u_j(n))$ and note that
\begin{align}\label{eq:drift}
  \E[\max\{\xi_{n+1} \cdot \theta,0\}  \> | \> \mathcal{F}_n]  = 
  & \tilde{u}_1 \max\{\theta_1(1-\tilde{u}_1)-\theta_2 \tilde{u}_2-\theta_3 \tilde{u}_3,0\}  + \nonumber \\
  & \tilde{u}_2 \max\{-\theta_1 \tilde{u}_1+\theta_2(1-\tilde{u}_2)-\theta_3 \tilde{u}_3),0\} + \nonumber \\
  & \tilde{u}_3 \max\{\theta_1 \tilde{u}_1-\theta_2 \tilde{u}_2+\theta_3(1-\tilde{u}_3),0\}.
\end{align}
Now, write $\theta = (\theta_1,\theta_2,\theta_3)$ with $\theta_1+\theta_2+\theta_3=0$ and $\theta_1^2+\theta_2^2+\theta_3^2=1$, and suppose that $\theta_i < 0$ for exactly two coordinates, without loss of generality $i=2,3$.  Then $\min{\theta_2,\theta_3}\leq -\frac{1}{\sqrt{6}}$, as otherwise $(-\theta_2-\theta_3)^2+\theta_2^2+\theta_3^2<1$.
Then we can write \[\theta_1(1-\tilde{u}_1)-\theta_2 \tilde{u}_2-\theta_3 \tilde{u}_3=-\theta_2(1-\tilde{u}_1+\tilde{u}_2)-\theta_3(1-\tilde{u}_1+\tilde{u}_3) \geq \frac{2a}{(1+2a)\sqrt{6}},\] as $\tilde{u}_1\leq \frac{1}{1+2a}$, meaning that the first term on the right hand side of \eqref{eq:drift} is at least $\frac{2a^2}{(1+2a)^2\sqrt{6}}$.
On the other hand, suppose that $\theta_i < 0$ for exactly one coordinate; without loss of generality assume $\theta_3 < 0$.  Then if $\theta_1>\theta_2$ we can write \[\theta_1(1-\tilde{u}_1)-\theta_2 \tilde{u}_2-\theta_3 \tilde{u}_3= \theta_1(1-\tilde{u}_1+\tilde{u}_3)-\theta_2(\tilde{u}_2-\tilde{u}_3).\]  If $\tilde{u}_3\geq \tilde{u}_2$ then this is at least $\theta_1(1-\tilde{u}_1+\tilde{u}_3)\geq \frac{2a}{(1+2a)\sqrt{6}}$, while if $\tilde{u}_3 < \tilde{u}_2$ it is at least $\theta_1(1-\tilde{u}_1-\tilde{u}_2+2\tilde{u}_3)$, which is equal to $3\theta_1\tilde{u}_3$ because $\tilde{u}_1+\tilde{u}_2+\tilde{u}_3=1$.  Hence in that case $\theta_1(1-\tilde{u}_1)-\theta_2 \tilde{u}_2-\theta_3 \tilde{u}_3\geq \frac{3a}{(1+2a)\sqrt{6}}$, so that again the first term on the right hand side of \eqref{eq:drift} is at least $\frac{3a^2}{(1+2a)^2\sqrt{6}}$.  If $\theta_2\geq \theta_1$ similar arguments show the the same but for the second term.  Hence we can set $\kappa=\frac{2a^2}{(1+2a)^2\sqrt{6}}>0$, and apply Theorem 1 of \cite{pemantle1990} to conclude the result.
\end{proof}

\section{Proofs for the cyclic case}\label{cyclic}

\subsection{Introduction}

In this section we prove Theorem \ref{cyclicmain}. Here $A$ is the matrix
$$\begin{pmatrix} 1 & 1 & 0 \\ 0 & 1 & 1 \\ 1 & 0 & 1 \end{pmatrix},$$ and we have \begin{eqnarray*}F_1(x_1,x_2,x_3) &=& \frac{x_1^{\beta}+x_2^{\beta}}{2(x_1^{\beta}+x_2^{\beta}+x_3^{\beta})}-x_1\\
F_2(x_1,x_2,x_3) &=& \frac{x_2^{\beta}+x_3^{\beta}}{2)(x_1^{\beta}+x_2^{\beta}+x_3^{\beta})}-x_2\\
F_3(x_1,x_2,x_3) &=& \frac{x_3^{\beta}+x_1^{\beta}}{2(x_1^{\beta}+x_2^{\beta}+x_3^{\beta})}-x_3\end{eqnarray*}

First, we note that for any choice of $\beta$, $(\frac13,\frac13,\frac13)$ is a stationary point of $F$.
The following two results give information on its stability and show that it is in fact the only stationary point.

\begin{lem}\label{stability}For the vector field $F$, the stationary point at $(\frac13,\frac13,\frac13)$ is stable if $\beta<4$, and a linearly unstable source if $\beta>4$. \end{lem}

\begin{proof}
As we are working with $x\in\Delta^2$, write $x_3=1-x_1-x_2$.  Routine calculus then shows that the Jacobian matrix at $(\frac13,\frac13,\frac13)$ is $$\begin{pmatrix}\frac{\beta}{2}-1 & \frac{\beta}{2} \\ -\frac{\beta}{2} &  -1 \end{pmatrix}.$$
The eigenvalues of this Jacobian are then the roots $\lambda$ of $$\lambda^2+\lambda\left(2-\frac{\beta}2\right)-\left(\frac{\beta}2-1\right)+\left(\frac{\beta}{2}\right)^2$$
which are $$\left(\frac{\beta}{4}-1\right)\pm \frac{i\sqrt{3}\beta }{4}.$$  Since the real part of both eigenvalues is positive if $\beta>4$ and negative if $\beta<4$, the result follows.
\end{proof}

\begin{lem}\label{statpoint}The only stationary point of $F$ in $\Delta^2$ is $(\frac13,\frac13,\frac13)$.\end{lem}

\begin{proof}
For $x=(x_1,x_2,x_3)\in \Delta^2$ with $F_1(x)=F_2(x)=F_3(x)=0$, we have $x_1=\frac{x_1^{\beta}+x_2^{\beta}}{x_2^{\beta}+x_3^{\beta}}x_2$ and similarly $x_2=\frac{x_2^{\beta}+x_3^{\beta}}{x_3^{\beta}+x_1^{\beta}}x_3$ and $x_3=\frac{x_3^{\beta}+x_1^{\beta}}{x_1^{\beta}+x_2^{\beta}}x_1$.  Using this, $$x_1-x_2=x_2\frac{x_1^{\beta}-x_3^{\beta}}{x_2^{\beta}+x_3^{\beta}},$$ indicating that (if $x_2>0$) if $x_1>x_2$ then also $x_1>x_3$, while if $x_1<x_2$ then $x_1<x_3$.  Similarly, if $x_3>0$ then the signs of $x_2-x_3$ and $x_2-x_1$ are the same, and if $x_1>0$ then the signs of $x_3-x_1$ and $x_3-x_2$ are the same.  Hence the only stationary point of $F$ in the interior of $\Delta^2$ is $(\frac13,\frac13,\frac13)$.

It is also easy to check that if $x_1=0$ then $x_2=0$, and similarly that if $x_2=0$ then $x_3=0$ and if $x_3=0$ then $x_1=0$.  Hence there are no stationary points of $F$ on the boundary of $\Delta^2$.
\end{proof}

We can now complete the proof of Theorem \ref{cyclicmain}.

That we only have one stationary point, and that it is never a saddle, restricts the possibilities for chain transitive sets.  In two dimensions Theorem 6.12 of Bena\"{i}m \cite{benaim} states that chain transitive sets must be unions of stationary points, periodic orbits and cyclic orbit chains.  However, with only one stationary point which is not a saddle cyclic orbit chains are impossible.  We can thus conclude that the limit set must be a connected union of periodic orbits and stationary points.

By Lemma \ref{stability}, if $\beta<4$ then by Lemma \ref{stability} the stationary point $(\frac13,\frac13,\frac13)$  is stable, and hence is an attractor for the flow given by $F$.  As in the proof of Proposition \ref{convergence}, by Theorem 7.3 of Bena\"{i}m \cite{benaim}, to show that there is positive probability of convergence to $(\frac13,\frac13,\frac13)$ it is enough to show that it is an attainable point, that is that we have that  $\pr[\exists ~ s \geq t ~:~ X(s) \in N] > 0$ for every neighborhood $N$ of $(\frac13,\frac13,\frac13)$. This is straightforward to show: for any $\epsilon>0$ there will be points of the form $\left(\frac{n_1}{n},\frac{n_2}{n},\frac{n_3}{n}\right)$ with $n_1,n_2,n_3$ integers satisfying $n_1+n_2+n_3=n$ and $\max\left\{\frac{n_1}{n}-\frac13,\frac{n_2}{n}-\frac13,\frac{n_3}{n}-\frac13\right\}<\epsilon$ for arbitrarily large $n$.  Because even if only one colour is present in the urn initially, there will be positive probability that a second colour is chosen at the first time step, and once two colours are present in the urn all three colours have positive probability of being chosen at each step, there will be positive probability of any such point being reached, so $(\frac13,\frac13,\frac13)$ is indeed attainable.

If $\beta>4$, then by Lemma \ref{stability} $(\frac13,\frac13,\frac13)$ is linearly unstable, so as in Proposition \ref{convergence} it will follow that it is a limit with probability zero if we have positive expectation of the positive part of the component of the noise in any given direction.  In fact, the same argument as in Proposition \ref{convergence} will work here, except that the explicit bounds for the $\tilde{u}_i$ must be replaced by the fact that in a neighbourhood of $(\frac13,\frac13,\frac13)$ there will exist $\delta$ such that $\delta<\tilde{u}_i<1-\delta$ for each $i$, giving $\kappa=\delta^2/\sqrt{6}$.  So we can conclude that $\bar{x}(n)$ has probability zero of converging to a stationary point, and it follows that the limit set must be a periodic orbit or a connected union of periodic orbits, completing the proof of Theorem \ref{cyclicmain}.

\section{Examples and simulations} \label{ex_sim}

In this section we consider some examples, including some where exact calculations are possible, and some simulations.  We also consider some examples which go beyond the cases covered by Theorems \ref{sym3main} and \ref{cyclicmain}.

\subsection{The symmetric case with $\beta=2$}
In the case of Theorem \ref{sym3main} with $\beta=2$, the possible limits and the phase transitions can be explicitly identified.  We find that $\mathcal{P}_{a,2}(z)=(z-1)(az^2+(a-1)+2a)$, with roots given by $z=1$ and $z=\frac{1-a \pm\sqrt{1-2a-7a^2}}{2a}$.  If $a< \frac{\sqrt{8}-1}{7}$, then these are real and positive, and letting $\lambda_1=\frac{3a+1-\sqrt{1-2a-7a^2}}{4(2a+1)}, \lambda_2=\frac{3a+1+\sqrt{1-2a-7a^2}}{4(2a+1)}, \lambda_3=\frac{a+1-\sqrt{1-2a-7a^2}}{2(2a+1)}$ and $\lambda_4=\frac{a+1+\sqrt{1-2a-7a^2}}{2(2a+1)}$, we obtain linearly stable stationary points $(\lambda_1,\lambda_1,\lambda_4)$, $(\lambda_1,\lambda_4,\lambda_1)$ and $(\lambda_4,\lambda_1,\lambda_1)$, and linearly unstable stationary (except at $a=\frac14$) points $(\lambda_2,\lambda_2,\lambda_3)$, $(\lambda_2,\lambda_3,\lambda_2)$, $(\lambda_3,\lambda_2,\lambda_2)$.  In addition, the stationary point $(\frac13,\frac13,\frac13)$ is linearly stable if $a>\frac14$, and linearly unstable if $a<\frac14$.

If $a>\frac{\sqrt{8}-1}{7}$, then $(\frac13,\frac13,\frac13)$ is the only stationary point, and is stable.  In the notation of Theorem \ref{sym3main}, we have $\beta_1\left(\frac{\sqrt{8}-1}{7}\right)=2$, and the three phases are as follows:

\begin{figure}[h]\label{fig:simulations-sym}
  \caption{20 simulations of the symmetric model for $\beta=2$}
  \centering
\begin{subfigure}{.33\textwidth}
  \centering
  \includegraphics[width=5cm, height=5cm]{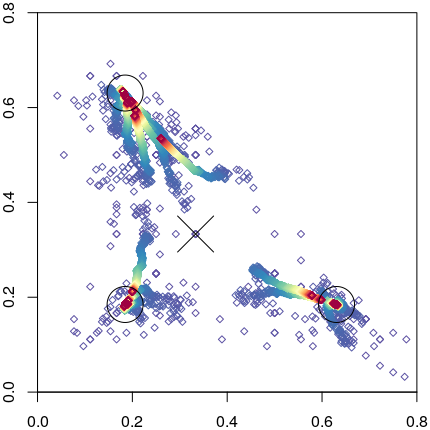}
  \caption{$a=0.2$}\label{ap2}
  \label{fig:sub1}
\end{subfigure}%
 \centering
\begin{subfigure}{.33\textwidth}
  \centering
  \includegraphics[width=5cm, height=5cm]{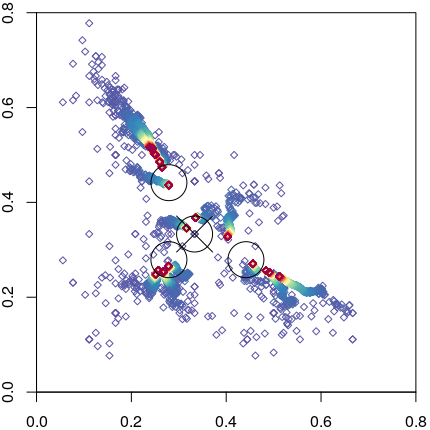}
  \caption{$a=0.26$}\label{ap26}
  \label{fig:sub2}
\end{subfigure}%
\centering
\begin{subfigure}{.33\textwidth}
  \centering
  \includegraphics[width=5cm, height=5cm]{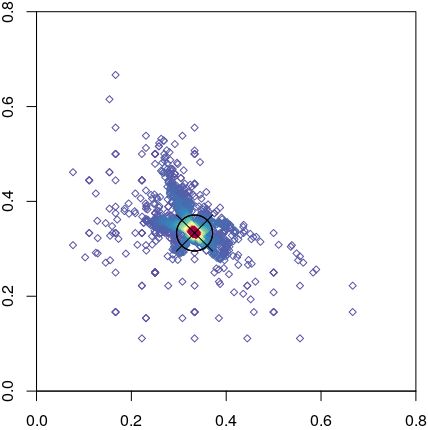}
  \caption{$a=0.5$}\label{ap5}
  \label{fig:sub3}
\end{subfigure}
\end{figure}

\begin{itemize}
\item When $a<\frac14$, $(\frac13,\frac13,\frac13)$ is linearly unstable; there are three other stationary points, $(\lambda_1,\lambda_1,\lambda_4)$ and permutations, placed symmetrically, which are stable.  For example when $a=\frac15$, $(0.1847,0.1847,0.6306)$ and permutations are stable.  Almost surely, one of these three points will be the limit.  A simulation of 20 trajectories of the stochastic process in this case appears in Figure \ref{ap2}.
\item  For $\frac14<a<\frac{\sqrt{8}-1}{7}$, $(\frac13,\frac13,\frac13)$ is now stable but there are also stable stationary points elsewhere, near $(\frac14,\frac14,\frac12)$.  In this case, both symmetric and asymmetric limits have positive probability.  For example, at $a=0.26$, there are stable stationary points at $(0.2792,0.2792,0.4416)$ and permutations as well as $(\frac13,\frac13,\frac13)$.  A simulation of 20 trajectories of the stochastic process in this case appears in Figure \ref{ap26}.
\item For $a>\frac{\sqrt{8}-1}{7}$, $(\frac13,\frac13,\frac13)$ is the only stationary point, and is stable, and will be the limit almost surely.  A simulation of 20 trajectories of the stochastic process in the $a=\frac12$ case appears in Figure \ref{ap5}.
\end{itemize}

At the critical value $a=\frac14$, $(\frac13,\frac13,\frac13)=(\lambda_2,\lambda_2,\lambda_3)$ has zeros as eigenvalues of its Jacobian and so is neither linearly stable nor linearly unstable, while there are stable stationary points at $(\frac14,\frac14,\frac12)$ and permutations; similarly at the critical value $a=\frac{\sqrt{8}-1}{7}$ the stationary point $(\lambda_2,\lambda_2,\lambda_3)=(\lambda_1,\lambda_1,\lambda_4)$ is neither linearly stable nor linearly unstable.

\subsection{The symmetric case with $\beta=3$}
It is also possible to do some explicit calculations when $\beta=3$.  In this case $(\frac13,\frac13,\frac13)$ is linearly stable when $a>\frac25$ and linearly unstable when $a<\frac25$, and we have $$\mathcal{P}_{a,3}(z)=(z-1)(az^3+(a-1)z^2+(a-1)z+2a),$$ where the cubic factor has one real root (which is negative) when $a>a_c=\frac{1}{166}(3.(2)^{1/2}.(3)^{1/4} + 24.(3)^{1/2} + 13.(2)^{1/2}.(3)^{3/4}-20) = 0.4160306$ and three real roots (one of which is negative) when $a<a_c$.  (In the notation of Theorem \ref{sym3main}, $\beta_1(a_c)=3$.)  Hence for $a>a_c$ we get almost sure convergence to $(\frac13,\frac13,\frac13)$, for $a<\frac25$ we get almost sure convergence to one of three asymmetric stationary points, and for $\frac25<a<a_c$ the process may converge either to $(\frac13,\frac13,\frac13)$ or to an asymmetric stationary point, each with positive probability.

\subsection{The cyclic model}

Figure \ref{cyclic3} shows 20 simulations of the cyclic model when $\beta=3$, showing convergence to  $(\frac13,\frac13,\frac13)$.
Figure \ref{cyclic6} shows 20 simulations with $\beta=6$, showing apparent convergence to a single limit cycle.

\begin{figure}[h]\label{fig:simulations-cyclic}
  \caption{20 simulations of the cyclic model}
  \centering
\begin{subfigure}{.33\textwidth}
  \centering
  \includegraphics[width=5cm, height=5cm]{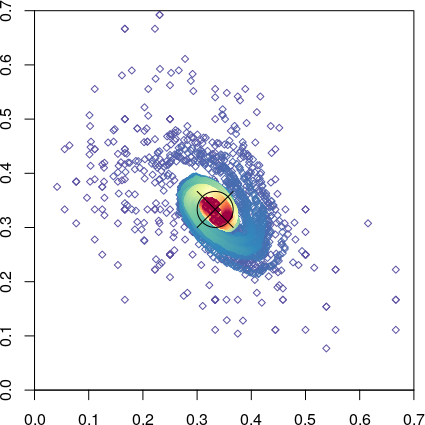}
  \caption{$\beta=3$}\label{cyclic3}
  \label{fig:sub1-cyclic}
\end{subfigure}%
 \centering
\begin{subfigure}{.33\textwidth}
  \centering
  \includegraphics[width=5cm, height=5cm]{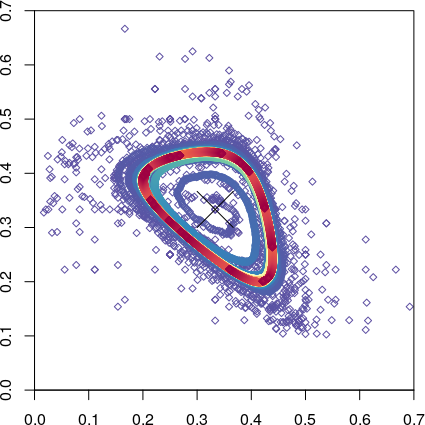}
  \caption{$\beta=6$}\label{cyclic6}
  \label{fig:sub2-cyclic}
\end{subfigure}%
\end{figure}

\subsection{The symmetric case with more than three types}\label{gthree}

It is natural to extend Section \ref{symmetric} to $d>3$ types, letting $A$ be the $d\times d$ matrix with $a_{ii}=1$ for each $i$ and $a_{ij}=a$ for $i\neq j$.  It is straightforward to extend the Lyapunov function \eqref{eq:lyapunov} to this case, meaning that Lemma \ref{point} applies.  However, the later calculations, starting with Proposition \ref{eqcoord}, do not apply.  It thus may be interesting to investigate whether more complex patterns of phases can occur in this case than when $d=3$; however, numerical solution of the stationary point equations for particular values of $a$ when $d=4,5,6$ suggests that the behaviour is in fact very similar to the $d=3$ case, with three phases which parallel those found in Theorem \ref{sym3main}.

\subsection{A more general cyclic case}

It is natural to extend Section \ref{cyclic} by allowing the matrix $A$ to take the form $$\begin{pmatrix} 1 & a & 0 \\ 0 & 1 & a \\ a & 0 & 1\end{pmatrix},$$ allowing for different strengths of the cyclic reinforcement.  It is straightforward to extend Lemma \ref{stability} to this case, showing that the stationary point at $(\frac13,\frac13,\frac13)$ is linearly stable if $a\geq 2$ or if $a<2$ and $\beta<\frac{2(1+a)}{2-a}$, and linearly unstable if $a<2$ and $\beta>\frac{2(1+a)}{2-a}$.  However Lemma \ref{statpoint} does not apply for general $a$ and there may be other stationary points.

Numerical investigation when $\beta=2$ suggests that there are three phases in $a$: in addition to a phase with almost sure convergence to $(\frac13,\frac13,\frac13)$ and a phase with convergence to a limit cycle, there is a phase up to $a \approx 0.25057$ where there are stable stationary points other than $(\frac13,\frac13,\frac13)$ and that the process usually converges to one of these.

\section*{Acknowledgements}

M.C.’s research was supported by CONICET [grant number 10520170300561CO].
The authors are grateful to Andrew Wade for suggestions that improved the presentation.

\bibliographystyle{plain}
\bibliography{references}

\end{document}